\documentclass[10pt,reqno]{amsart}

\usepackage[a4paper,left=35mm,right=35mm,top=30mm,bottom=30mm,marginpar=25mm]{geometry}

\usepackage{amsmath}
\usepackage{amssymb}
\usepackage{amsthm}
\usepackage[latin1]{inputenc}
\usepackage{eurosym}
\usepackage[dvips]{graphics}
\usepackage{graphicx}
\usepackage{epsfig}
\usepackage{hyperref}
\usepackage{dsfont}
\usepackage[nocompress, space]{cite}
\usepackage{xcolor}

\usepackage{mathtools}
\mathtoolsset{showonlyrefs}

\usepackage[displaymath,mathlines]{lineno}

\allowdisplaybreaks

\usepackage{ifthen}

\makeindex

\newcommand{\tr}{\operatorname{tr}}

\newcommand{\supp}{\operatorname{supp}}

\renewcommand{\div}{\operatorname{div}}

\newcommand{\Rr}{{\mathbb{R}}}

\newcommand{\Tt}{{\mathbb{T}}}

\newcommand{\Ll}{{\mathcal{L}}}

\newcommand{\Ii}{{\mathcal{I}}}

\newcommand{\Pp}{{\mathcal{P}}}

\def\leq{\leqslant}
\def\geq{\geqslant}

\numberwithin{equation}{section}

\newtheoremstyle{thmlemcorr}{10pt}{10pt}{\itshape}{}{\bfseries}{.}{10pt}{{\thmname{#1}\thmnumber{
#2}\thmnote{ (#3)}}}
\newtheoremstyle{thmlemcorr*}{10pt}{10pt}{\itshape}{}{\bfseries}{.}\newline{{\thmname{#1}\thmnumber{
#2}\thmnote{ (#3)}}}
\newtheoremstyle{defi}{10pt}{10pt}{\itshape}{}{\bfseries}{.}{10pt}{{\thmname{#1}\thmnumber{
#2}\thmnote{ (#3)}}}
\newtheoremstyle{remexample}{10pt}{10pt}{}{}{\bfseries}{.}{10pt}{{\thmname{#1}\thmnumber{
#2}\thmnote{ (#3)}}}
\newtheoremstyle{ass}{10pt}{10pt}{}{}{\bfseries}{.}{10pt}{{\thmname{#1}\thmnumber{
A#2}\thmnote{ (#3)}}}

\theoremstyle{thmlemcorr}
\newtheorem{theorem}{Theorem}
\numberwithin{theorem}{section}
\newtheorem{lemma}[theorem]{Lemma}

\newtheorem{proposition}[theorem]{Proposition}

\theoremstyle{thmlemcorr*}
\newtheorem{theorem*}{Theorem}
\newtheorem{lemma*}[theorem]{Lemma}
\newtheorem{corollary*}[theorem]{Corollary}
\newtheorem{proposition*}[theorem]{Proposition}
\newtheorem{problem*}[theorem]{Problem}
\newtheorem{conjecture*}[theorem]{Conjecture}

\theoremstyle{defi}
\newtheorem{definition}[theorem]{Definition}

\theoremstyle{remexample}
\newtheorem{remark}[theorem]{Remark}

\newtheorem{cor}[theorem]{Corollary}

\theoremstyle{ass}

\begin{document}

\title[Displacement convexity for first-order mean-field games]{Displacement convexity for first-order mean-field games}

\author{Diogo A. Gomes}
\address[D. A. Gomes]{
        King Abdullah University of Science and Technology (KAUST), CEMSE Division , Thuwal 23955-6900. Saudi Arabia.}
\email{diogo.gomes@kaust.edu.sa}

\author{Tommaso Seneci}
\address[T. Seneci]{
        King Abdullah University of Science and Technology (KAUST), CEMSE Division , Thuwal 23955-6900. Saudi Arabia.}
\email{tommaso.seneci@kaust.edu.sa}


\thanks{
        D. Gomes  and T. Seneci were partially supported by KAUST
baseline funds and KAUST OSR-CRG2017-3452.
}
\date{\today}

\begin{abstract}
Here, we consider the planning problem for first-order mean-field games (MFG). When there is no coupling between players, MFG degenerate into optimal transport problems. 
Displacement convexity is a fundamental tool in optimal transport that often reveals hidden convexity
of functionals and, thus, has numerous applications in the calculus of variations. 
We explore the similarities between the Benamou-Brenier formulation of optimal transport and MFG to extend displacement convexity methods from to MFG. In particular, 
we identify a class of functions, that depend on solutions of MFG, that are convex in time and, thus, 
obtain new a priori bounds for solutions of MFG.
A remarkable consequence is the log-convexity of $L^q$ norms. This convexity gives bounds for the density of solutions of the planning problem and extends displacement convexity of $L^q$ norms from optimal transport. Additionally, we prove the convexity of $L^q$ norms for MFG with congestion.
\end{abstract}

\maketitle

\section{Introduction}

Displacement convexity is an alternative concept of convexity used often in minimization problems in spaces of measures. Displacement convexity was introduced in  \cite{mccann1997convexity}  to study a non-convex variation problem where it revealed a hidden convexity that gives existence and uniqueness of minimizers.

Given two probability measures $\mu,\nu\in \Pp(\Rr^d)$, we say that a map $T:\Rr^d\to \Rr^d$
transports $\mu$ into $\nu$ if $\nu =T_\# \mu$, where
\begin{equation*}
\int_{\Rr^d} f(x)dT_\# \mu(x) = \int_{\Rr^d} f(T(x))d\nu(x) 
\end{equation*}
for all bounded continuous $f:\Rr^d\to \Rr$. 
In optimal transport, we are given two probability measures $\mu,\nu\in \Pp(\Rr^d)$, and we seek to transport $\mu$ into $\nu$ in the most efficient way according to  a given transport cost, 
see for example the surveys \cite{Santambrogiobook}, \cite{Villanithebook}, and\cite{Villani2008}. While this problem is discrete in nature, a remarkable alternative formulation due to Benamou and Brenier \cite{BB}, looks instead at paths in $\Pp(\Rr^d)$ that connect $\mu$ to $\nu$. The Benamou-Brenier formulation of optimal transport consists of minimizing the energy functional
\begin{equation*}
\int_{\Rr^d}\int_0^1 \rho^t(x) |v(x,t)|^2dxdt,
\end{equation*}
over all smooth velocity fields $v(x,t)$, with trajectories $T_x(t)$, and densities $\rho^t = T_{(\cdot)}(t)_\# \mu$, such that $\rho^0=\mu$ and $\rho^1 = \nu$. Under suitable regularity conditions, the optimality conditions of this variational problem are
\begin{equation} \label{BBintro}
\begin{cases}
-u_t + \frac{|Du|^2}{2} = \bar{H} \\
(\rho^t)_t - \div(\rho^t Du) = 0 \\
\rho^t \in \Pp_{ac}(\Rr^d) \quad \forall\ t\in [0,1] \\
\rho^0=\mu, \ \rho^1 = \nu,
\end{cases}
\end{equation}
where $v(x,t) = - D_x u(x,t) $ and $\bar{H}\in \Rr$. The \emph{displacement interpolant} between $\mu$ and $\nu$ is the minimizer of the Benamou-Brenier problem. A functional, $\mathcal{F}:\Pp(\Rr^d)\to \Rr$, is {\em displacement convex} if $t\mapsto \mathcal{F}(\rho^t)$ is convex for all {displacement interpolants} $\rho^t$. By using \eqref{BBintro}, we can differentiate twice $\mathcal{F}(\rho^t)$ to study displacement convexity. This methodology was used in \cite{Schachter2018}, where the author identifies a new class of displacement convex functionals that depend on spatial derivatives of the density.

Mean-field games (MFG) model the interaction between identical rational agents, see the original papers in \cite{Caines2, Caines1} and  \cite{ll1, ll2,ll3}, or the surveys \cite{cardaliaguet, GS,bensoussan,  gomes2016regularity}. In these games, each agent minimizes a value function, which is the same for every agent. In classical MFG, agents choose their trajectories given an initial configuration and a terminal cost. In the MFG planning problem \cite{LCDF, CDY, porretta}, the initial and terminal distribution of the agents are prescribed while the terminal cost is unknown. Here, we focus on the planning problem for first-order MFG. These games are given by a Hamilton-Jacobi equation coupled with a continuity equation
\begin{equation} \label{MFGintro}
\begin{cases}
-u_t + H(Du) = g(m) \\
m_t - \div(m D_p H(Du)) = 0 \quad& \forall\ (x,t)\in \Tt^d\times (0,T) \\
m(\cdot,0)=m^0(\cdot),\ m(\cdot,T)=m^T(\cdot).
\end{cases}
\end{equation}
Here, we use periodic boundary conditions, thus the spatial domain for \eqref{MFGintro} is $\Tt^d$, the $d$-dimensional torus. A classical solution of \eqref{MFGintro} is a pair $(u(x,t),m(x,t))\in C^\infty(\Tt^d \times [0,T])$, such that $u(x,t)$ and $m(x,t)\geq 0$ are periodic in $x$ for all time $t\in [0,T]$. The function $m$ represents the statistical distribution of the agents in space, whereas $u$ represents their value function. 
The Hamiltonian, 
$H(Du)$, 
takes into accounts the movement cost of the agents and their preferred direction of motion, and $g(m)$ determines the interactions between agents.
As can be seen by comparing \eqref{BBintro} with \eqref{MFGintro}, the optimal transport problem is a special case of a first-order MFGs where the interaction between the agents does not exist; that is, $g=0$. 
In the {\em initial-terminal value} problem, \eqref{MFGintro} is endowed with initial, $m(\cdot,0)=m^0(\cdot)$, and terminal, $u(\cdot,T)=u^T(\cdot)$, conditions; that is, agents are given a terminal cost, and their initial distribution is specified. In contrast, in the {\em planning problem}, $m^0$ and $m^T$, the initial and terminal distributions, are specified. Thus, our goal is to find a cost, $u$, that steers agents from an initial distribution, $m^0$, to a desired terminal distribution, $m^T$.

The initial-terminal value problem for second-order MFGs is now well understood. 
The existence and uniqueness of smooth solutions of the time dependent problem for 
were first studied in \cite{ll2, ll3}, and examined in detail in \cite{LCDF}. Subsequently, several authors considered
classical \cite{GPM2, GPM3, Gomes2016c, Gomes2015b} and weak solutions \cite{CarLasLionsPorr12, porretta2}. 
For first-order MFGs, several Sobolev regularity results were obtained in  \cite{GraCard}, \cite{cgbt}, \cite{San16}, and \cite{graber2017sobolev}. 
The planning problem was addressed from a variational numerical perspective in 
\cite{CDY} and for
second-order MFGs in \cite{porretta, porretta2}. 

Here, we explore displacement convexity  
properties to construct a new class of estimates for first-order MFGs. 
In particular, 
the primary goal of this paper is to identify functions $U:\Rr^+_0 \to \Rr$ such that
\begin{equation} \label{functionalConvex}
t \mapsto \int U(m(x,t))dx \quad \text{is convex},
\end{equation}
where $m(x,t)$ solves \eqref{MFGintro}. In the case of optimal transport, 
\eqref{functionalConvex} is displacement convex if
\begin{equation} \label{DispConvMFGintro}
\begin{cases}
P(z) = U'(z) z - U(z), \\
P\in C^1(\Rr^+_0),\ P(z)\geq 0, \\
\frac{P(z)}{z^{1-\frac{1}{d}}} \quad \text{non-decreasing}.
\end{cases}
\end{equation}
The convexity of the preceding functional gives the following 
a priori bound:
\[
\int U(m(x,t))dx\leq \frac{t}{T }\int U(m(x,T))dx
+\left(1-\frac{t}{T }\right)\int U(m(x,0))dx, 
\]
which are particularly interesting in the case of the planning problem because $m(x,0)$ and $m(x,T)$ are known. 

In Section \ref{SMFG}, we prove the following result on the convexity of functionals that depend on the density of solutions of first-order MFGs, as in \eqref{functionalConvex}.
\begin{theorem}
\label{mainthm}
Let $m,u\in C^\infty(\Tt^d\times [0,T])$, $m\geq 0$, be periodic solutions of the first-order MFG
\begin{equation} \label{MFGMain}
\begin{cases}
-u_t+H(Du) = g(m) \\
m_t-\div(m D_pH(Du)) = 0
\end{cases}
\end{equation}
with $g:\Rr^+_0\to \Rr$, $H:\Rr^d\to \Rr$ smooth, $g$ non-decreasing, and $H$ convex. If $U:\Rr^+_0\to \Rr$ is such that \eqref{DispConvMFGintro} holds, then
\begin{equation*}
t \mapsto \int_{\Tt^d} U(m(x,t))dx \quad \text{is convex.}
\end{equation*}
\end{theorem}
Functionals of the form $t \mapsto \int_{\Tt^d} m(x,t)^q dx$ satisfy the conditions of the preceding theorem. Moreover, a careful computation reveals the convexity of $t\to \ln(\|m(\cdot,t)\|_{L^q(\Tt^d)})$ for all $1\leq q \leq \infty$, see Proposition \ref{functionalConvex}. Furthermore, this log-convexity generalizes the  result in  \cite{mccann1997convexity} concerning the  displacement convexity of $\rho \mapsto \int \rho(x)^q dx$.
Here, we should also mention the recent work \cite{LavSantambrogio2017} where, using a discretization method and ideas that are reminiscent from displacement convexity, the authors prove additional regularity for mean-field games.

MFGs with congestion model 
the case where the agents' displacement cost increases in high-density regions. These games
correspond to the system
\begin{equation*}
\begin{cases}
-u_t + m^\alpha H\left( \frac{Du}{m^\alpha} \right) = g(m) \\
m_t - \div\left(m D_p H\left( \frac{Du}{m^\alpha} \right) \right) = 0 \quad & \forall\ (x,t)\in \Tt^d\times (0,T) \\
m(\cdot,t)\in \Pp_{ac}(\Rr^d) & \forall\ t\in(0,T) \\
m(\cdot,0)=m^0(\cdot),\ m(\cdot,T)=m^T(\cdot)
\end{cases}
\end{equation*}
for $\alpha>0$. 
The existence and uniqueness of solutions of second-order classical MFG with congestion were studied  in \cite{EvGom,Gomes2015b} in the stationary case and 
 in \cite{Achdou2016,Graber2} in the non-stationary case. First-order MFG with congestion were studied in the stationary case in \cite{EFGNV2017} ,  \cite{nurbekyan17}, and  \cite{GNPr216}  and in the
 time-dependent case, for the forward-forward model, in \cite{GomesSedjroFFMFGCongestion}. 
 In particular, as far as the authors are aware the planning problem was not studied previously. 
Here, 
in Section \ref{SExtension}, we examine  the case where $H(p)=\frac{|p|^\beta}{\beta}$ and, in Theorem \ref{THMcongestion},  prove the convexity of $t\mapsto \int_{\Tt^d} m(x,t)^pdx$, $p$ depending on $\alpha$ and $\beta$. As an application, we obtain $L^\infty$ bounds for the density
in Corollary \ref{remarkCongestion}.

\section{Preliminaries}
\label{SdispConv}

Here, we briefly review the optimal transport problem and the Benamou-Brenier formulation. Subsequently, we recall displacement convexity and discuss elementary examples.

\subsection{Optimal Transport}

Let $\Pp(\Rr^d)$ be the set of probability measures in $\Rr^d$, and $\Pp_{ac}(\Rr^d)$ the subset of those probabilities that are absolutely continuous with respect to the Lebesgue measure.

The optimal transport problem, also known as the Monge-Kantorovich problem,
studies the optimal way of moving mass between two different locations. 
We are given an initial distribution of mass determined by a probability measure, $\mu\in \Pp(\Rr^d)$, and a target distribution given by another probability measure, $\nu \in \Pp(\Rr^d)$. To each unit of mass moved from $x\in \Rr^d$ to $y\in \Rr^d$, 
we associate a cost, $c(x,y)$. The Monge-Kantorovich problem consists of minimizing the total cost,
\begin{equation} \label{MKproblem}
\int_{\Rr^d\times \Rr^d} c(x,y)d\pi(x,y),
\end{equation}
over the set of plans $\Pi[\mu,\nu] = \{\pi \in \Pp(\Rr^d\times \Rr^d) \text{ with marginals } \mu,\nu \}$. If $c(x,y)$ is positive and lower semi-continuous, there exist a minimizer of \eqref{MKproblem}, see \cite{Villanithebook}. For a quadratic cost, $c(x,y) = |x-y|^2$, a duality formulation due to Kantorovich uncovers remarkable properties of the optimal plan.  If $\mu,\nu\in \Pp_{ac}(\Rr^d)$ and have finite second-order moments, the minimizing plan is unique and can be written in the form
\begin{equation*}
\pi = (Id \times D\phi)_\# \mu,
\end{equation*}
where $D\phi$ is the unique gradient of a convex function such that $\nu = D\phi_\# \mu$; that is, for every $E\subset \Rr^d$ measurable,
\begin{equation*}
\nu(E) = (D\phi_\# \mu)(E) = \mu((D\phi)^{-1}(E)).
\end{equation*}
Thus, the minimum of \eqref{MKproblem} equals to
\begin{equation} \label{MKminimized}
\int_{\Rr^d\times \Rr^d} |x-D\phi(x)|^2d\mu(x)dy.
\end{equation}
In the literature, $D\phi$ is called the Brenier's map transporting $\mu$ into $\nu$, see \cite{MR1100809}.

\subsection{The Benamou-Brenier Formulation}

In the time-dependent optimal transport problem, each particle moves from $\mu$ to $\nu$ according to a piecewise $C^1$ trajectory $T_x(t):[0,1]\to \Rr^d$. At $t=0$, $T_x(0)=x$, and at time $t=1$ particles reach their destination in $\supp(\nu)$. Accordingly, we require $\nu = T_{(\cdot)}(1)_\# \mu$. The time-dependent optimal transport problem consists of minimizing a displacement cost  $C = C(T_x(\cdot))$ over all trajectories $\{T_x(\cdot)\}_{x\in\Rr^d}$ transporting $\mu$ into $\nu$, i.e.
\begin{equation*}
\inf \left\{ \int_X C(\{T_x(t)\}_{t\in (0,1)})d\mu(x): T_x(0) = x,\ {T_{(\cdot)}(1)}_\# \mu = \nu \right\}.
\end{equation*}
An important case is given by differential cost function
\begin{equation*}
C(\{T_x(t)\}_{t\in (0,1)}) = \int_0^1 c(\dot{T}_x(t))dt,
\end{equation*}
where $c$ is a convex function. Thanks to Jensen's inequality, we find
\begin{equation} \label{compare1}
\int_0^1 c(\dot{T}_x(t))dx \geq c\left( \int_0^1 \dot{T}_x(t) dx \right) = c(y-x).
\end{equation}
For $c(x)=|x|^2$, by comparing \eqref{compare1} with \eqref{MKminimized}, we see that straight lines are admissible trajectories. Hence, they are minimizers. Thus, the optimal velocities are $x-D\phi(x)$, $D\phi(x)$ being the Brenier's map transporting $\mu$ into $\nu$. This means that the minimizing straight lines are
\begin{equation} \label{displInterpolantsTrajectory}
T_x(t) = (1-t)x + tD\phi(x).
\end{equation}
At each time $t\in [0,1]$, $\mu$ is transported into
\begin{equation*}
\rho^t = ((1-t)x+tD\phi(x))_\# \mu.
\end{equation*}


The previous discussion suggests we move our perspective to the Eulerian point of view. For that, we fix $x\in\Rr^d$ and  consider a smooth trajectory $T_x(t)$ determined by a Lipschitz velocity field $v(x,t)$; that is,
\begin{equation*} 
\begin{cases}
\dot{T}_x(t) = v(T_x(t),(t)) \\
T_x(0) = x.
\end{cases}
\end{equation*}

If $\{T_{(\cdot)}(t)\}_{0\leq t \leq T}$ is a Lipschitz family of diffeomorphisms, the pushforward $\rho^t = {T_{(\cdot)}(t)}_\# \mu$ is the unique solution of the \emph{continuity equation}
\begin{equation} \label{MassTransportPDE}
\frac{\partial \rho^t}{\partial t} + \div(\rho^t v) = 0
\end{equation}
in the weak sense. We look for a path $\rho^t$ that minimizes the total action
\begin{equation*}
A[\rho,v] = \int_0^1 E(t)dt = \int_0^1 \int_{\Rr^d} \rho^t(x) \frac{|v(x,t)|^2}{2}dxdt.
\end{equation*}
As in \cite{BB}, if $\mu,\nu\in P_{ac}(\Rr^d)$ are compactly supported and satisfy suitable conditions \cite{Villanithebook}, then
\begin{equation*}
\inf_{\pi\in \Pi[\mu,\nu]} \int_{\Rr^d\times\Rr^d} |x-y|^2 d\pi(x,y) = \inf_{\rho,v} A[\rho,v],
\end{equation*}
where the infimum on the r.h.s is taken over all smooth $(\rho,v)$ solving \eqref{MassTransportPDE} with $\rho^0=\mu$ and $\rho^1=\nu$. The optimality conditions of this problem correspond to \eqref{BBintro}, as we describe next. We already have a partial differential equation solved by the density $\rho^t$. Moreover, if $\{T_x(\cdot)\}_{x\in\Rr^d}$ are constant speed trajectories such that $x\mapsto T_x(t)$ are diffeomorphisms for all $t$, then $v(x,t)$ solves
\begin{equation*} 
\frac{\partial v}{\partial t} + v\cdot Dv = 0.
\end{equation*}
Because of \eqref{displInterpolantsTrajectory},  it turns out that $v$ is a gradient. Thus, $v=-D_x u$. Consequently, $u$ solves a Hamilton-Jacobi equation
\begin{equation} \label{HamJacBase}
-\frac{\partial u}{\partial t} + \frac{|Du|^2}{2} = \bar{H}, \ \bar{H}\in \Rr.
\end{equation}
If we combine \eqref{HamJacBase} with \eqref{MassTransportPDE}, we get the following system
\begin{equation} \label{almostMFG}
\left\{ \begin{array}{rl}
-\frac{\partial u}{\partial t} + \frac{|Du|^2}{2}& = \bar{H} \\
\frac{\partial \rho^t}{\partial t} - \div (\rho^t Du ) &= 0.
\end{array}\right.
\end{equation}
Because displacement interpolants are constant speed trajectories, \eqref{almostMFG} are the corresponding optimality conditions.

The system \eqref{almostMFG} has a triangular structure. The first equation does not depend on $m$, while  the second one depends on $Du$. First-order MFGs are recovered by adding a coupling $g=g(m)$ to the Hamilton-Jacobi equation. Due to this coupling, MFGs no longer have this triangular structure, and, thus, their study becomes substantially more challenging.

\subsection{Displacement Convexity}

Displacement convexity was introduced in \cite{mccann1997convexity} to study a non-convex variational problem from the theory of interacting gases. In that problem, the gas density  is determined by a probability, $\rho\in P_{ac}(\Rr^d)$. Each particle is subject to two  forces: one given by an interaction potential $W(x-y)$ that increases with the distance between particles, and the other given by the internal energy, $U(z)$. The potential is 
\begin{equation*}
\mathcal{W}(\rho) = \frac{1}{2} \int_{\Rr^d\times \Rr^d} W(x-y)d\rho(x)d\rho(y),
\end{equation*}
and the internal energy is
\begin{equation*}
\mathcal{U}(\rho) = \int_{\Rr^d} U(\rho(x))dx.
\end{equation*}
In that model, the configuration of the gas minimizes the energy
\begin{equation*} 
E(\rho) = \mathcal{U}(\rho) + \mathcal{W}(\rho).
\end{equation*}
Given the variational nature of the problem, the convexity of $E$ is of paramount importance. If $U$ is convex, then $\mathcal{U}$ is also convex. However, convexity of $W$ does not imply the
convexity of $\mathcal{W}$.

A fundamental contribution in \cite{mccann1997convexity} is a new way of interpolating two probabilities densities, $\mu,\nu\in \Pp_{ac}(\Rr^d)$, that reveals a hidden convexity in $\mathcal{U}$ and $\mathcal{W}$.
For a given family of trajectories $\{T_{(\cdot)}(t)\}_{t\in(0,1)},\ T_{(\cdot)}(t): \Rr^d\to \Rr^d$, we set $\rho^t = {T_{(\cdot)}(t)}_\# \rho$. Thus, 
\begin{align*}
\mathcal{W}(\rho^t) &= \iint W(x-y)d\rho^t(x)d\rho^t(y) \\
&= \iint W(x-y)d ({T_{(\cdot)}(t)}_\# \rho)(y) d ({T_{(\cdot)}(t)}_\# \rho)(x) = \iint W(T_{x}(t)-T_{y}(t))d\rho (y)d\rho(x).
\end{align*}
Therefore, if $T_x(t)$ is linear in time, the map $t\mapsto \mathcal{W}(\rho^t)$ is convex. 
\begin{definition}
Let $\mu,\nu\in \Pp_{ac}(\Rr^d)$ and $D\phi :\Rr^d\to \Rr^d$ the unique gradient of a convex function such that $\nu = D\phi_\# \mu$. The \emph{displacement interpolant} between $\mu$ and $\nu$ is
\begin{equation*}
\rho^t = ((1-t)(\cdot) + t D\phi(\cdot))_\# \mu.
\end{equation*}
\end{definition}
A
 function $\mathcal{F}:\Pp_{ac}(\Rr^d)\to \Rr$
 is \emph{displacement convex}  if it 
 is convex along displacement interpolants; that is,
\begin{equation*}
t\mapsto \mathcal{F}(\rho^t) \ \text{is convex for all $\rho^t$ displacement interpolants}.
\end{equation*}

As we have seen, the map $t\mapsto \mathcal{W}(\rho^t)$ is convex; that is, $\rho\mapsto \mathcal{W}(\rho)$ is displacement convex. In general, even if $U$ is convex, $\mathcal{U}$ may not be displacement convex. However, the following condition proven in \cite{mccann1997convexity} gives the required convexity: if
\begin{equation} \label{InternalEnergy}
z \mapsto z^d U(z^{-d}), z\in \Rr^+ \quad \text{ is convex, non-increasing and }\ U(0)=0,
\end{equation}
then 
\begin{equation*}
t\mapsto \mathcal{U}(\rho^t) = \int U(\rho^t(x))dx \quad \text{is convex}.
\end{equation*}
 In \cite{Villanithebook}, the author derives conditions equivalent to
 \eqref{InternalEnergy} for $U$ sufficiently smooth in terms
of the pressure
\begin{equation*}
P(z) = U'(z)z-U(z).
\end{equation*}
By differentiating twice $z\mapsto z^d U(z^{-d})$ and using the preceding identity into the resulting expression, we conclude that for $U\in C^1(\Rr^+_0)$ and if $P$ satisfies \eqref{DispConvMFGintro}. 
then $\rho \mapsto \int U(\rho)$ is displacement convex. Notice that $P'$ is non-negative, as we recover by differentiating $\frac{P(z)}{z^{1-\frac{1}{d}}}$,
\begin{equation*} 
z P'(z) \geq \left( 1-\frac{1}{d} \right)P(z) \geq 0.
\end{equation*}
Consequently, we can differentiate $P$ to show that the above condition implies the convexity of  $U$:
\begin{equation*}
P'(z) = U''(z)z+U'(z)-U'(z)=U''(z)z \geq 0.
\end{equation*}

Here, we use an alternative approach explored in \cite{Schachter2018} to study displacemnet convexity. 
Formally, because displacement interpolants are solutions of the Benamou-Brenier problem, \eqref{almostMFG}, to check displacement convexity, it is enough to prove that
\begin{equation} \label{convexFunPrelim}
\frac{d^2}{dt^2} \int U(\rho^t(x))dx \geq 0
\end{equation}
for all $(\rho^t,u)$ smooth solutions of \eqref{almostMFG}. 
Because first-order MFGs are recovered by coupling the Hamilton-Jacobi equation in \eqref{almostMFG}, differentiating \eqref{convexFunPrelim} for $\rho^t \equiv m$ solving \eqref{almostMFG} may lead to similar displacement convexity inequalities. In the next section, we prove that this holds provided that the coupling $g(m)$ is non-decreasing and $H(p)$ is convex.

\section{Displacement convexity in first-order mean-field games}
\label{SMFG}

Here, we prove that, if $U$ satisfies \eqref{DispConvMFGintro}, then $t\mapsto \int_{\Tt^d} U(m(x,t))dx$ is convex, where $(u, m)$ solves \eqref{MFGMain}.
We end this section by examining the one-dimensional case, where more precise results can be proven. 

\subsection{Convex Functionals for First-Order Mean-Field Games}

Here, we prove our main result, Theorem \ref{mainthm}, that extends displacement convexity to MFG.

\begin{proof}[Proof of Theorem \ref{mainthm}]We begin by the following computation
\begin{align*}
\frac{d}{d t} \int U(m) &= \int U'(m)m_t = \int U'(m)\div(mD_pH) \\
&= \int U'(m)m\div(D_pH) + U'(m) Dm D_pH \\
&= \int U'(m)m \div(D_pH) + D(U(m)) D_pH \\
&= \int U'(m)m\div(D_pH) - U(m) \div(D_pH) = \int P(m)\div(D_pH),
\end{align*}
where $P(m)$ is given by \eqref{DispConvMFGintro}. Differentiating again, we obtain
\begin{align*}
\frac{d^2 }{d t^2} \int U(m) &= \int P'(m)m_t \div(D_pH) + P(m)\div(D_p H_t ) \\
& = \int P'(m)\div(m D_pH) \div(D_p H) + P(m) \div(D_{pp}^2 H Du_t) \\
&= \int \overbrace{P'(m)m \div(D_pH)^2}^A + \overbrace{P'(m) Dm D_pH \div(D_pH)}^B \\
&  + \overbrace{P(m)\div(D_{pp}^2H D(H))}^C - \overbrace{P(m)\div(g'(m) D_{pp}^2H Dm)}^D.
\end{align*} 

We want to generalize Lemma 5.43 in \cite{Villanithebook}; thus, we expect to get an inequality of the form
\begin{align*} 
\frac{d^2}{d t^2} \int U(m) &\geq  \int \overbrace{mP'(m)\div(D_pH)^2}^{A'} -\overbrace{P(m)\div(D_pH)^2}^{B'} + \overbrace{\frac{P(m)}{d} \div(D_pH)^2}^{C'} dx \\
& \hspace{6cm} + \ \text{ non-negative terms}.
\end{align*}
The first three terms, $A',B'$ and $C'$, correspond to the optimal transport case; that is, $g=0$. Due to conditions \eqref{DispConvMFGintro}, $A'+B'+C'\geq 0$. We note that $A=A'$. Next, we integrate by parts $B$ to get
{\small \begin{align} \label{equationQ}
B &= \int P'(m) Dm D_pH \div(D_p H) = \int D(P(m)) D_p H \div(D_p H) \\
&= -\int P(m)\div(D_p H \div(D_p H)) = \int - \overbrace{P(m)\div(D_p H)^2}^{B'} - \overbrace{P(m)D_p H D(\div(D_pH))}^Q.  \notag
\end{align} }
To simplify $C$, we compute
\begin{align}\label{divergenceTrace}
\div(D_{pp}^2 H D(H)) &= \div(D_{pp}^2 H D^2u D_pH) = \div(D(D_pH)D_pH) \\\notag
&= ((H_{p_i})_{x_j} H_{p_j})_{x_i} = (H_{p_i})_{x_i,x_j} H_{p_j} + (H_{p_i})_{x_j} (H_{p_j})_{x_i} \\\notag
&= D(\div(D_p H))D_p H + \tr((D(D_p H))^2). 
\end{align}
Then, we expand $C$ as follows
\begin{equation*}
C = \int P(m)\div(D_{pp}^2H D(H)) = \int \overbrace{P(m)D(\div(D_p H))D_p H}^Q + P(m)\tr((D(D_p H))^2)
\end{equation*}
and notice that $Q$ cancels the corresponding term in \eqref{equationQ}. Finally, $D$ is
\begin{equation*}
D = \int (-P(m)\div(D_{pp}^2 H D(g(m)))) = \int P'(m)g'(m) Dm D_{pp}^2 H Dm. 
\end{equation*}
According to the preceding identities, we get
\begin{align*}
\frac{d^2}{dt^2} \int U(m) =& \int \overbrace{P'(m)m \div(D_pH)^2}^{A'} - \overbrace{P(m) \div(D_pH)^2}^{B'} \\
& + P(m) \tr((D(D_pH))^2) + P'(m) g'(m) Dm D_{pp}^2 H Dm.
\end{align*}
Because $D(D_pH) = D_{pp}H D^2 u$ is the product of a positive semidefinite matrix and a symmetric matrix, Lemma \eqref{lemmaMatrix} implies
\begin{equation*}
\tr((D(D_pH))^2) \geq \frac{1}{d} \tr(D(D_pH))^2 = \frac{1}{d} \div(D_pH)^2.
\end{equation*}
Since $P$ is non-negative, we obtain
\begin{equation*}
\int P(m)\tr((D(D_pH))^2) \geq \int \overbrace{\frac{1}{d}P(m)\div(D_pH)^2}^{C'}.
\end{equation*}
Finally, in view of the preceding identities, the last expression becomes
\begin{align} \label{estimateMain}
\frac{d^2}{dt^2} \int U(m) \geq \int & \left( P'(m)m - P(m) + \frac{1}{d}P(m) \right)\div(D_pH)^2 \\\notag
& + P'(m) g'(m) Dm D_{pp}^2 H Dm \geq 0,
\end{align}
which is convex because \eqref{DispConvMFGintro} holds, because $g'(m)\geq 0$, and because $H(p)$ is convex.
\end{proof}

\subsection{$L^q$ Estimates}

In the previous section, we identified conditions on $U$ such that $t \mapsto \int U(m(x,t))dx$ is convex  when $(u,m)$ solves a first-order MFG. The function $U(z) = z^q$ satisfies \eqref{DispConvMFGintro} for all $1 \leq q < \infty$. Here, we refine this result and prove the log-convexity of the $L^q$ norms of the density. 

\begin{proposition} \label{LogConvexity}
Let $u,m\in C^\infty(\Tt^d\times [0,T])$ be periodic solutions of \eqref{MFGMain} with $g,H$ smooth, $g$ non-decreasing, and $H$ convex. Then, for all $1\leq q \leq \infty$,
\begin{equation} \label{ineqNorms}
\|m(\cdot,t)\|_{L^q(\Tt^d)} \leq \|m^0(\cdot)\|_{L^q(\Tt^d)}^{1-\frac{t}{T}}\|m^T(\cdot)\|_{L^q(\Tt^d)}^{\frac{t}{T}}, \quad \forall\ t\in [0,T].
\end{equation}
\end{proposition}
\begin{proof} First of all, notice that if $f$ is smooth and positive, then $\ln f$ is convex if and only if
\begin{equation} \label{logConvexity}
(\ln f)'' = \left(\frac{f'}{f}\right)' = \frac{f'' f - (f')^2}{f^2} \geq 0;
\end{equation}
that is,
\begin{equation} \label{logConvexityForm}
\quad f'' f \geq (f')^2.
\end{equation}
First, we consider the case $1\leq q<\infty$. We begin by computing $P(z)=U'(z)z-U(z)=(qz^{q-1})z-z^q = (q-1)z^q$. Then, plug $U(z)=z^q$ into \eqref{estimateMain} to get
\begin{align*}
\frac{d^2}{dt^2} \int m(x,t)^q \geq &\int \left( q-1+\frac{1}{d}\right)(q-1)m^q\div(D_pH)^2 + q(q-1)m^{q-1} g'(m) Dm D_{pp}^2 H Dm \\
\geq & (q-1)^2 \int m^q \div(D_pH)^2.
\end{align*}
Thus,
\begin{align*}
\left( \frac{d}{dt} \int m^q \right)^2 &= \left( (q-1)\int m^q \div(D_pH) \right)^2 \\
&\leq (q-1)^2 \left(\int m^q \right)\left(\int m^q\div(D_pH)^2\right) \leq \left( \int m^q \right) \left( \frac{d^2}{dt^2} \int m^q \right).
\end{align*}
The preceding inequality combined with \eqref{logConvexityForm} shows that $\ln\left(\int m^q\right)$ is convex. Therefore,
\begin{align*}
\ln\left(\int m(x,t)^q\right) &\leq \left(1-\frac{t}{T}\right) \ln\left( \int m^0(x)^q \right) + \frac{t}{T} \ln\left( \int m^T(x)^q \right) \\
&= \ln\left( \left( \int m^0(x)^q \right)^{ 1-\frac{t}{T} } \left( \int m^T(x)^q \right)^{\frac{t}{T}} \right).
\end{align*}
Therefore, 
\begin{equation*}
\int m(x,t)^q \leq \left( \int m^0(x)^q \right)^{ 1-\frac{t}{T} } \left( \int m^T(x)^q \right)^{\frac{t}{T}}.
\end{equation*}
Exponentiating the previous inequality to $\frac{1}{q}$ to obtain the result.

Finally, we address the case $q=\infty$. Because $\Ll^d(\Tt^d)=1<\infty$, we can pass to the limit in \eqref{ineqNorms} as $q\to\infty$ to derive the estimate for the supremum.
\end{proof}

\begin{remark}
For $g(m)=0$ and $H(p) = \frac{|p|^2}{2} + \bar{H}$, $\bar{H}\in \Rr$, solutions of \eqref{MFGMain} are displacement interpolants between the initial density, $m^0$, and the terminal density, $m^T$. Therefore, Proposition \eqref{LogConvexity} gives both the log convexity of $L^q$ norms and $L^\infty$ bounds for the optimal transport problem, provided the initial and terminal densities are bounded. 
\end{remark}

For certain choices of $g$ and $H$, 
the preceding estimate can be improved even further if $1<q<\infty$, For example, here, we examine the case $g'(m)\geq C m^\alpha, C>0, \alpha\in\Rr$ and $H$ uniformly convex. When $\alpha<0$, we assume $m(x,t)>0$ everywhere.

\begin{lemma}
Let $m(x,t)$ be as in Theorem \eqref{mainthm} and suppose that
\begin{equation*}
\int_{\Tt^d} m(x,t)dx = 1
\end{equation*}
for all $t\in [0,T]$. Assume also that $\|m^0\|_{L^q}, \|m^T\|_{L^q}>1, g'(m)\geq C m^\alpha, C>0, \alpha\in\Rr$ and $H$ is uniformly convex. Then, for all $1<q<\infty$,
\begin{equation*}
\|m(\cdot,t)\|_{L^q(\Tt^d)} < \|m^0(\cdot)\|_{L^q(\Tt^d)}^{1-\frac{t}{T}}\|m^T(\cdot)\|_{L^q(\Tt^d)}^{\frac{t}{T}}, \quad \forall\ t\in (0,T).
\end{equation*}
\end{lemma}
\begin{proof}
We select $f(t)=\int m(x,t)^q$, to which corresponds $P(z)=(q-1)z^q$,  and use \eqref{estimateMain} and \eqref{logConvexity} to get the inequality
\begin{align*}
\frac{d^2}{dt^2} \ln\left( \int m^q \right) & \geq \frac{\left(\int m^q\right)\left( \int P'(m)g'(m) Dm D_{pp}^2 H Dm\right)}{\left(\int m^q\right)^2} = C \frac{\int m^{q-1+\alpha} |Dm|^2}{\int m^q} \\
&= \begin{cases}
C \frac{\int |D(m^{\frac{q+1+\alpha}{2}})|^2}{\int m^q}, \quad & \alpha\neq -q-1 \\
C \frac{\int |D\ln(m)|^2}{\int m^q}, \quad & \alpha = -q-1 \\
\end{cases} 
\end{align*}
Because $m$ integrates to $1$, Jensen's inequality implies $ \|m(\cdot,t)\|_{L^q(\Tt^d)}=1$ if and only if $m(x,t)=1$ for all $x\in \Tt^d$. Therefore, $\|m(\cdot,\bar{t})\|_{L^p}>1$ if and only if $\|m(\cdot,t)\|_{L^q(\Tt^d)}$ is strictly convex in a neighborhood of $\bar{t}$. Because $\|m^0 \|_{L^p}>0$, then $t\mapsto \|m(\cdot,t)\|_{L^p}$ is strictly convex in a neighborhood of $0$. Analogously, $t\mapsto \|m(\cdot,t)\|_{L^p}$ is strictly convex in a neighborhood of $t=T$. Therefore, the inequality in \eqref{ineqNorms} is strict for all $t\in (0,T)$.
\end{proof}

\subsection{Convexity in dimension 1} \label{Remark1dimension}
Finally, we address the one-dimensional case, $d=1$. A direct computation shows that the convexity of $U$ implies the convexity of $t\mapsto \int_0^1 U(m(x,t))dx$. Accordingly, convexity holds for functions of the form $U(z) = (z+\varepsilon)^{-q}$, $q\geq 0, \varepsilon>0$; that is,
\begin{equation*}
\int_0^1 \frac{1}{(m(x,t)+\varepsilon)^q}dx \leq \left(1-\frac{t}{T}\right) \int_0^1 \frac{1}{(m^0(x)+\varepsilon)^q}dx + \frac{t}{T} \int_0^1 \frac{1}{(m^T(x)+\varepsilon)^q}dx.
\end{equation*}
Now, raising both sides to the power $\frac{1}{q}$ and bounding the r.h.s, we get
\begin{align*}
\| (m(\cdot,t)+\varepsilon)^{-1} \|_{L^q} &\leq \left( \left(1-\frac{t}{T}\right) \int_0^1 \frac{1}{(m^0(x)+\varepsilon)^q}dx + \frac{t}{T} \int_0^1 \frac{1}{(m^T(x)+\varepsilon)^q}dx \right)^{\frac{1}{p}} \\
& \leq \max\left\{ \int_0^1 \frac{1}{(m^0(x)+\varepsilon)^q}dx, \int_0^1 \frac{1}{(m^T(x)+\varepsilon)^q}dx \right\}^{\frac{1}{q}} \\
&= \max\{ \|(m^0(\cdot)+\varepsilon)^{-1} \|_{L^q}, \|(m^T(\cdot)+\varepsilon)^{-1} \|_{L^q} \}.
\end{align*}
By letting $\varepsilon\to 0$ and then $q\to\infty$, we get
\begin{equation*}
\| m(\cdot,t)^{-1} \|_{L^\infty} \leq \max\{ \|m^0(\cdot)^{-1} \|_{L^\infty}, \|m^T(\cdot)^{-1} \|_{L^\infty} \}.
\end{equation*}
Finally, we invert the above inequality to get quasi-concavity for the infimum
\begin{equation*}
\inf m(\cdot,t) \geq \min\{ \inf m^0(\cdot), \inf m^T(\cdot)\}.
\end{equation*}

\section{Extension to First-Order MFG with congestion}
\label{SExtension}

In MFG with congestion, the Hamiltonian-Jacobi equation depends on the inverse of the density, $m(x,t)$. Here, we study MFGs with Hamiltonians $H(p) = \frac{|p|^\beta}{\beta}$ and with a congestion exponent $\alpha>0$.
\begin{theorem} \label{THMcongestion}
Let $m,u\in C^\infty(\Tt^d\times [0,T])$, $m>0$, be periodic solutions of the first-order MFG with congestion
\begin{equation} \label{MFGwithCongestion}
\begin{cases}
-u_t + m^{\alpha(1-\beta)} \frac{|Du|^\beta}{\beta} = g(m) \\
m_t - \div(m^{1+\alpha(1-\beta)} Du |Du|^{\beta-2}) = 0 & (x,t) \in \Tt^d\times (0,T) \\
\end{cases}
\end{equation}
with $g:\Rr^+\to \Rr$ smooth and non-decreasing. If
\begin{equation}
\beta\geq 2,\ \ q+2\alpha(1-\beta)\geq 0 \quad  \text{and }\ 1 - \frac{1 - \frac{1}{d}}{q+ 2\alpha(1-\beta)} - \frac{\alpha(\beta-1)}{2} \geq 0         \label{convexityCongestion} 
\end{equation}
or
\begin{equation}
1 < \beta<2,\ \ q+ 2\alpha(1-\beta) \geq 0 \quad  \text{and} \ 1 - \frac{1 - \frac{1}{d}}{q+ 2\alpha(1-\beta)} - \frac{\alpha}{2} \geq 0,           \label{convexityCongestion1}
\end{equation}
then
\begin{equation} \label{convexityCong}
t\mapsto \int_{\Tt^d} m(x,t)^q dx \quad \text{ is convex}.
\end{equation}
\end{theorem}
\begin{proof}
First, we compute
\begin{align*}
\frac{d}{dt} \int m^q &= q \int m^{q-1} m_t = q \int m^{q-1} \div(m^{1+\alpha(1-\beta)} Du |Du|^{\beta-2}) \\
&= q\int m^{q-1} \Big((1+\alpha(1-\beta)) m^{\alpha(1-\beta)} Dm Du |Du|^{\beta-2} + m^{1+\alpha(1-\beta)} \div(Du |Du|^{\beta-2}) \Big) \\
&= q \int (1+\alpha(1-\beta)) m^{q-1+\alpha(1-\beta)} Dm Du |Du|^{\beta-2} + m^{q+\alpha(1-\beta)} \div(Du |Du|^{\beta-2}) \\
&= q \int \frac{(1+\alpha(1-\beta))}{q+\alpha(1-\beta)} D(m^{q+\alpha(1-\beta)}) Du |Du|^{\beta-2} + m^{q+\alpha(1-\beta)} \div(Du |Du|^{\beta-2}) \\
&= q \int \left( 1 -  \frac{1+\alpha(1-\beta)}{q+\alpha(1-\beta)} \right) m^{q+\alpha(1-\beta)} \div(Du |Du|^{\beta-2}) \\
&= \int \frac{q(q-1)}{q+\alpha(1-\beta)}  m^{q+\alpha(1-\beta)} \div(Du |Du|^{\beta-2}).
\end{align*}
Thus, we have
\begin{equation*}
\frac{1}{q(q-1)} \frac{d^2}{dt^2} \int m^q = \int \overbrace{m^{q+\alpha(1-\beta)-1} m_t \div(Du|Du|^{\beta-2})}^A + \overbrace{\frac{1}{q+\alpha(1-\beta)} m^{q+\alpha(1-\beta)} \div((Du |Du|^{\beta-2})_t)}^B.
\end{equation*}
Now, we expand and integrate by parts $A$
\begin{align*}
A=& \int m^{q+\alpha(1-\beta)-1} m_t \div(Du|Du|^{\beta-2}) \\
=& \int m^{q+\alpha(1-\beta)-1} \div(m^{1+\alpha(1-\beta)} Du |Du|^{\beta-2}) \div(Du|Du|^{\beta-2}) \\
=& \int m^{q+\alpha(1-\beta)-1} \div(Du|Du|^{\beta-2}) \Big((1+\alpha(1-\beta)) m^{\alpha(1-\beta)} Dm Du |Du|^{\beta-2} + m^{1+\alpha(1-\beta)} \div(Du |Du|^{\beta-2}) \Big) \\
=& \int (1+\alpha(1-\beta)) m^{q+ 2\alpha(1-\beta)-1} Dm Du |Du|^{\beta-2} \div(Du|Du|^{\beta-2}) + m^{q+2\alpha(1-\beta)} \div(Du |Du|^{\beta-2})^2 \\
=& \int \frac{(1+\alpha(1-\beta))}{q+ 2\alpha(1-\beta)} D(m^{q+ 2\alpha(1-\beta)}) Du |Du|^{\beta-2} \div(Du|Du|^{\beta-2}) + m^{q+2\alpha(1-\beta)} \div(Du |Du|^{\beta-2})^2 \\
=& \int -\frac{(1+\alpha(1-\beta))}{q+ 2\alpha(1-\beta)} m^{q+ 2\alpha(1-\beta)} \div\Big(Du |Du|^{\beta-2} \div(Du|Du|^{\beta-2})\Big) + m^{q+2\alpha(1-\beta)} \div(Du |Du|^{\beta-2})^2 \\
=& \int \left( 1-\frac{(1+\alpha(1-\beta))}{q+ 2\alpha(1-\beta)} \right) m^{q+ 2\alpha(1-\beta)} \div(Du |Du|^{\beta-2})^2 \\
&\hspace{3cm} -\frac{(1+\alpha(1-\beta))}{q+ 2\alpha(1-\beta)} m^{q+ 2\alpha(1-\beta)} D(\div(Du|Du|^{\beta-2})) Du |Du|^{\beta-2}. \\
\end{align*}
Before simplifying $B$, we compute $(Du |Du|^{\beta-2})_t$:
\begin{align*}
(Du |Du|^{\beta-2})_t &= |Du|^{\beta-2} Du_t + Du (\beta-2) |Du|^{\beta-4} Du Du_t \\
&= (\Ii |Du|^{\beta-2} +(\beta-2) |Du|^{\beta-4} Du\otimes Du) Du_t = R Du_t,
\end{align*}
where $\Ii$ is the identity matrix and  $R = (\Ii |Du|^{\beta-2} +(\beta-2) |Du|^{\beta-4} Du\otimes Du) = D_{pp}H$. Using the preceding identity, we expand $B$ as follows
\begin{align*}
B =& \frac{1}{q+\alpha(1-\beta)} \int m^{q+\alpha(1-\beta)} \div((Du |Du|^{\beta-2})_t) = - \int m^{q+\alpha(1-\beta)-1} Dm  (Du|Du|^{\beta-2})_t \\
=& - \int m^{q+\alpha(1-\beta)-1} Dm R Du_t = - \int m^{q+\alpha(1-\beta)-1} Dm R \left( D\left(m^{\alpha(1-\beta)} \frac{|Du|^\beta}{\beta} \right) - g'(m)Dm \right) \\
=& \int \overbrace{-m^{q+\alpha(1-\beta)-1} Dm R D\left(m^{\alpha(1-\beta)} \frac{|Du|^\beta}{\beta} \right)}^C + \overbrace{m^{q+\alpha(1-\beta)-1} g'(m) Dm R Dm}^D.
\end{align*}
Because $\beta\geq 1$, we get $D\geq 0$ as follows
\begin{align*}
D =& \int m^{q+\alpha(1-\beta)-1} g'(m) Dm (\Ii |Du|^{\beta-2} +(\beta-2) |Du|^{\beta-4} Du\otimes Du) Dm \\
=& \int m^{q+\alpha(1-\beta)-1} g'(m) (|Dm|^2 |Du|^{\beta-2} +(\beta-2) |Du|^{\beta-4} |DuDm|^2) \\
\geq& \int m^{q+\alpha(1-\beta)-1} g'(m) (|DmDu|^2 |Du|^{\beta-4} +(\beta-2) |Du|^{\beta-4} |DuDm|^2) \\
=& \int (\beta-1) m^{q+\alpha(1-\beta)-1} g'(m) |DmDu|^2 |Du|^{\beta-4},
\end{align*}
where we used $|DmDu|\leq |Dm||Du|$. Concerning $C$, we  expand further the expression
\begin{align*}
C=& -\int m^{q+\alpha(1-\beta)-1} Dm R D\left(m^{\alpha(1-\beta)} \frac{|Du|^\beta}{\beta} \right) \\
=& -\int m^{q+\alpha(1-\beta)-1} Dm R \left( \alpha(1-\beta) m^{\alpha(1-\beta)-1} \frac{|Du|^\beta}{\beta} Dm + m^{\alpha(1-\beta)} |Du|^{\beta-2} D^2u Du \right) \\
=& \int \overbrace{\alpha(\beta-1) m^{q+2\alpha(1-\beta)-2} Dm R Dm \frac{|Du|^\beta}{\beta}}^E - \overbrace{m^{q+2\alpha(1-\beta)-1} Dm R D^2u Du |Du|^{\beta-2}}^F.
\end{align*}
The first term above simplifies to
\begin{align*}
E =& \int \alpha(\beta-1) m^{q+2\alpha(1-\beta)-2} Dm (\Ii |Du|^{\beta-2} +(\beta-2) |Du|^{\beta-4} Du\otimes Du) Dm \frac{|Du|^\beta}{\beta} \\
=& \int \frac{\alpha(\beta-1)}{\beta} m^{q+2\alpha(1-\beta)-2} |Dm|^2 |Du|^{2\beta-2} +\frac{\alpha(\beta-1)(\beta-2)}{\beta} m^{q+2\alpha(1-\beta)-2} |DmDu|^2 |Du|^{2\beta-4} \\
\geq& \int \frac{\alpha(\beta-1)^2}{\beta} m^{q+2\alpha(1-\beta)-2} |DmDu|^2 |Du|^{2\beta-4}.
\end{align*}
Finally, $F$ becomes
\begin{align*}
F =& \int - m^{q+2\alpha(1-\beta)-1} Dm \Big(\Ii |Du|^{\beta-2} +(\beta-2) |Du|^{\beta-4} Du\otimes Du \Big) D^2u Du |Du|^{\beta-2} \\
=& \int -\frac{1}{q+2\alpha(1-\beta)} D(m^{q+2\alpha(1-\beta)}) \Big( D^2u Du |Du|^{2\beta-4} + (\beta-2) Du\otimes Du D^2u Du |Du|^{2\beta-6} \Big) \\
=& \int \frac{1}{ q+2\alpha(1-\beta)} m^{q+2\alpha(1-\beta)} \div(D^2u Du |Du|^{2\beta-4} + (\beta-2) Du\otimes Du D^2u Du |Du|^{2\beta-6}) \\
=& \int \frac{1}{ q+2\alpha(1-\beta)} m^{q+2\alpha(1-\beta)} \div\Big( \Big(D^2u |Du|^{\beta-2} + (\beta-2) |Du|^{\beta-4} Du \otimes (D^2u Du) \Big)Du|Du|^{\beta-2} \Big) \\
=& \int \frac{1}{ q+2\alpha(1-\beta)} m^{q+2\alpha(1-\beta)} \div\Big( D(Du|Du|^{\beta-2}) Du|Du|^{\beta-2} \Big) \\
=& \int \frac{1}{ q+2\alpha(1-\beta)} m^{q+2\alpha(1-\beta)} \tr( D(Du|Du|^{\beta-2})^2) \\
&\hspace{3cm} + \frac{1}{ q+2\alpha(1-\beta)} m^{q+2\alpha(1-\beta)} D(\div(Du|Du|^{\beta-2}) Du |Du|^{\beta-2}), \\
\end{align*}
where we used \eqref{divergenceTrace}. We add all terms and simplify, to conclude
\begin{align}
\frac{d^2}{dt^2} \int m^q \geq & \int \left( 1-\frac{(1+\alpha(1-\beta))}{q+ 2\alpha(1-\beta)} \right) m^{q+ 2\alpha(1-\beta)} \div(Du |Du|^{\beta-2})^2 \notag\\
&-\frac{(1+\alpha(1-\beta))}{q+ 2\alpha(1-\beta)} m^{q+ 2\alpha(1-\beta)} D(\div(Du|Du|^{\beta-2})) Du |Du|^{\beta-2} \notag\\
& + \frac{1}{ q+2\alpha(1-\beta)} m^{q+2\alpha(1-\beta)} \tr( D(Du|Du|^{\beta-2})^2) \notag\\
& + \frac{1}{ q+2\alpha(1-\beta)} m^{q+2\alpha(1-\beta)} D(\div(Du|Du|^{\beta-2}) Du |Du|^{\beta-2}) \notag\\
& + \frac{\alpha(\beta-1)^2}{\beta} m^{q+2\alpha(1-\beta)-2} |DmDu|^2 |Du|^{2\beta-4} \notag\\&+ (\beta-1) m^{q+\alpha(1-\beta)-1} g'(m) |DmDu|^2 |Du|^{\beta-4} \notag\\
\geq & \int \left( 1-\frac{(1+\alpha(1-\beta) - \frac{1}{d})}{q+ 2\alpha(1-\beta)} \right) m^{q+ 2\alpha(1-\beta)} \div(Du |Du|^{\beta-2})^2  \label{expressionG}\\
& + \overbrace{\frac{\alpha(\beta-1)}{p+ 2\alpha(1-\beta)} m^{q+ 2\alpha(1-\beta)} D(\div(Du|Du|^{\beta-2})) Du |Du|^{\beta-2}}^G \notag\\
& + \frac{\alpha(\beta-1)^2}{\beta} m^{q+2\alpha(1-\beta)-2} |DmDu|^2 |Du|^{2\beta-4}\notag \\&+ (\beta-1) m^{q+\alpha(1-\beta)-1} g'(m) |DmDu|^2 |Du|^{\beta-4}\notag,
\end{align}
where we used \eqref{lemmaMatrix} to estimate
\begin{equation*}
\tr( D(Du|Du|^{\beta-2})^2)  \geq \frac{1}{d} \div(Du|Du|^{\beta-2})^2.
\end{equation*}
We only need to bound $G$ from below. So, we integrate it by parts and use Cauchy-Schwartz inequality to conclude that
\begin{align}
G =& \int \frac{\alpha(\beta-1)}{q+ 2\alpha(1-\beta)} m^{q+ 2\alpha(1-\beta)} Du |Du|^{\beta-2} D(\div(Du|Du|^{\beta-2}))\notag \\
=& \int - \frac{\alpha(\beta-1)}{q+ 2\alpha(1-\beta)} \div( m^{q+ 2\alpha(1-\beta)} Du |Du|^{\beta-2} ) \div(Du|Du|^{\beta-2})\notag \\
=& \int - \frac{\alpha(\beta-1)}{q+ 2\alpha(1-\beta)} m^{q+ 2\alpha(1-\beta)} \div(Du|Du|^{\beta-2})^2   \label{continuationBeta} \\
& - \alpha(\beta-1) m^{q+ 2\alpha(1-\beta)} (m^{-1} Dm Du |Du|^{\beta-2} ) \div(Du|Du|^{\beta-2})\notag \\
\geq & \int -\frac{\alpha(\beta-1)}{q+ 2\alpha(1-\beta)} m^{q+ 2\alpha(1-\beta)} \div(Du|Du|^{\beta-2})^2\notag\\& - \frac{\alpha(\beta-1)}{2} m^{q+ 2\alpha(1-\beta)-2} |Dm Du|^2 |Du|^{2\beta-4} \notag\\
& -\frac{\alpha(\beta-1)}{2} m^{q+ 2\alpha(1-\beta)} \div(Du|Du|^{\beta-2})^2\notag \\
=& \int \left( \frac{\alpha(1-\beta)}{q+ 2\alpha(1-\beta)} - \frac{\alpha(\beta-1)}{2} \right) m^{q+ 2\alpha(1-\beta)} \div(Du|Du|^{\beta-2})^2 \notag\\\notag
& - \frac{\alpha(\beta-1)}{2} m^{q+ 2\alpha(1-\beta)-2} |Dm Du|^2 |Du|^{2\beta-4}.
\end{align}
We use this last term into \eqref{expressionG} to get
\begin{align*}
& \int \left( 1-\frac{(1+\alpha(1-\beta) - \frac{1}{d})}{q+ 2\alpha(1-\beta)} \right) m^{q+ 2\alpha(1-\beta)} \div(Du |Du|^{\beta-2})^2 \\
&+ \left( \frac{\alpha(1-\beta)}{q+ 2\alpha(1-\beta)} - \frac{\alpha(\beta-1)}{2} \right) m^{q+ 2\alpha(1-\beta)} \div(Du|Du|^{\beta-2})^2 \\
& - \frac{\alpha(\beta-1)}{2} m^{q+ 2\alpha(1-\beta)-2} |Dm Du|^2 |Du|^{2\beta-4} + \frac{\alpha(\beta-1)^2}{\beta} m^{q+2\alpha(1-\beta)-2} |DmDu|^2 |Du|^{2\beta-4} \\
& + (\beta-1) m^{q+\alpha(1-\beta)-1} g'(m) |DmDu|^2 |Du|^{\beta-4} \\
=& \int \left( 1 - \frac{1 - \frac{1}{d}}{q+ 2\alpha(1-\beta)} - \frac{\alpha(\beta-1)}{2} \right) m^{q+ 2\alpha(1-\beta)} \div(Du|Du|^{\beta-2})^2 \\
& + \frac{\alpha(\beta-1)(\beta-2)}{2\beta} m^{q+ 2\alpha(1-\beta)-2} |Dm Du|^2 |Du|^{2\beta-4} + (\beta-1) m^{q+\alpha(1-\beta)-1} g'(m) |DmDu|^2 |Du|^{\beta-4} .
\end{align*}
From the above inequality, we see that if \eqref{convexityCong} holds, then $\frac{d^2}{dt^2} \int m^q\geq 0$. This concludes the proof for the case $\beta\geq 2$.

For $1 < \beta<2$, we revisit the expression for $G$ in \eqref{continuationBeta}. Then, we modify the Cauchy-Schwartz inequality to include the term $(\beta-1)$; that is,
\begin{align*}
& \int - \frac{\alpha(\beta-1)}{q+ 2\alpha(1-\beta)} m^{q+ 2\alpha(1-\beta)} \div(Du|Du|^{\beta-2})^2 \\
& - \alpha m^{q+ 2\alpha(1-\beta)} ((\beta-1) m^{-1} Dm Du |Du|^{\beta-2} ) \div(Du|Du|^{\beta-2}) \\
\geq & \int -\frac{\alpha(\beta-1)}{q+ 2\alpha(1-\beta)} m^{q+ 2\alpha(1-\beta)} \div(Du|Du|^{\beta-2})^2 \\
&- \frac{\alpha(\beta-1)^2}{2} m^{q+ 2\alpha(1-\beta)-2} |Dm Du|^2 |Du|^{2\beta-4} -\frac{\alpha}{2} m^{q+ 2\alpha(1-\beta)} \div(Du|Du|^{\beta-2})^2. \\
=& \int \left( \frac{\alpha(1-\beta)}{q+ 2\alpha(1-\beta)} - \frac{\alpha}{2} \right) m^{q+ 2\alpha(1-\beta)} \div(Du|Du|^{\beta-2})^2 \\
& - \frac{\alpha(\beta-1)^2}{2} m^{q+ 2\alpha(1-\beta)-2} |Dm Du|^2 |Du|^{2\beta-4}.
\end{align*}
Thus, we have
\begin{align*}
\frac{1}{q(q-1)} \int m^q\geq & \int \left( 1 - \frac{1 - \frac{1}{d}}{q+ 2\alpha(1-\beta)} - \frac{\alpha}{2} \right) m^{q+ 2\alpha(1-\beta)} \div(Du|Du|^{\beta-2})^2 \\
& + \frac{\alpha(\beta-1)^2(2-\beta)}{2\beta} m^{q+ 2\alpha(1-\beta)-2} |Dm Du|^2 |Du|^{2\beta-4} \\
&+ (\beta-1) m^{q+\alpha(1-\beta)-1} g'(m) |DmDu|^2 |Du|^{\beta-4},
\end{align*}
which is non-negative if \eqref{convexityCongestion1} holds.
\end{proof}

The idea used in the case without congestion to prove log-convexity of $L^q$ norms fails in this case. However, we can still prove quasi-convexity of the $L^\infty$ norm. In the next corollary, we identify couples $(\alpha,\beta)\in \Rr^+ \times (1,\infty)$ such that the set of values $q\geq 1$ satisfying either \eqref{convexityCongestion} or \eqref{convexityCongestion1} is unbounded. Subsequently, we obtain a uniform bound on the density of solutions of \eqref{MFGwithCongestion}.

\begin{cor}\label{remarkCongestion}
Let $(u,m)$ be classical solutions of \eqref{MFGwithCongestion}. If
\begin{align*}
\beta\geq 2 \quad & \text{and} \quad \alpha < \frac{2}{\beta-1} \\
\text{or} \\
1 < \beta < 2\quad &\text{and} \quad \alpha<2,
\end{align*}
then, for every $d\geq 1$,
\begin{equation} \label{upperBoundCong}
\|m(\cdot,t)\|_{L^\infty(\Tt^d)} \leq \max\{ \|m(\cdot)^0\|_{L^\infty(\Tt^d)}, \|m(\cdot)^T\|_{L^\infty(\Tt^d)} \}.
\end{equation}
\end{cor}
\begin{proof}
First, we examine the case $\beta\geq 2$. Because $\alpha<\frac{2}{\beta-1}$, there is $\varepsilon>0$ such that $\alpha = \frac{2}{\beta-1}(1-\varepsilon)$. We use this value to simplify the following expression,
\begin{equation*}
1 - \frac{1 - \frac{1}{d}}{q+ 2\alpha(1-\beta)} - \frac{\alpha(\beta-1)}{2} = \varepsilon - \frac{1 - \frac{1}{d}}{q+ 2\alpha(1-\beta)}.
\end{equation*}
As $q\to \infty$, the r.h.s of the above identity converges to $\varepsilon$ for all $d\geq 1$. Thus, there exists $Q>0$ such that, for all $q>Q$,
\begin{equation*}
1 - \frac{1 - \frac{1}{d}}{q+ 2\alpha(1-\beta)} - \frac{\alpha(\beta-1)}{2} >0.
\end{equation*}
Moreover, upon taking $Q$ large enough, we can assume that $q+2\alpha(1-\beta) >0$ for all $q>Q$. By Theorem \eqref{THMcongestion}, $t\mapsto \int m(x,t)^qdx$ is convex for all $q>Q$. Following similar computations to \eqref{Remark1dimension}, we get
\begin{equation*}
\|m(\cdot,t)\|_{L^\infty(\Tt^d)} \leq \max\{ \|m(\cdot)^0\|_{L^\infty(\Tt^d)}, \|m(\cdot)^T\|_{L^\infty(\Tt^d)} \}.
\end{equation*}

Analogously, in the case $1 < \beta < 2$, we use \eqref{convexityCongestion1}, set $\alpha = 2(1-\varepsilon)$, and follow the same reasoning as for $\beta \geq 2$ to obtain \eqref{upperBoundCong}.
\end{proof}

\begin{remark}
In dimension $d=1$, because
\begin{equation*}
\frac{1 - \frac{1}{d}}{q+ 2\alpha(1-\beta)} = 0
\end{equation*}
for all $q> 2\alpha(\beta-1)$, we do not need to use the $\varepsilon$ argument. Therefore, \eqref{upperBoundCong} holds even in the case $\beta\geq 2$ and $\alpha = \frac{2}{\beta-1}$ or $1\leq \beta<2$ and $\alpha=2$.
\end{remark}

\appendix
\section{Appendix}

Here, we prove the lemma used in the proofs of theorems \ref{mainthm} and \ref{THMcongestion}.

\begin{lemma} \label{lemmaMatrix}
Let $A,B\in \Rr^{d\times d}$ be symmetric matrices with $A$ positive semidefinite, then
\begin{equation*}
\tr\Big((AB)^2\Big) \geq \frac{1}{d} \tr(AB)^2.
\end{equation*}
\end{lemma}
\begin{proof}
Notice that, if $A$ is symmetric positive semidefinite, then $A_\varepsilon = A + \varepsilon \Ii$ is symmetric positive definite and converges to $A$ as $\varepsilon\to 0$. By approximation, it is then enough to prove the lemma for $A$ positive definite.

Since $A$ is symmetric positive definite, it admits an invertible square root $A^{\frac{1}{2}}$. Multiplying $AB$ by $A^{-\frac{1}{2}}$ on the left and by $A^{\frac{1}{2}}$ on the right, we conclude that
\begin{equation*}
A^{-\frac{1}{2}} (AB) A^{\frac{1}{2}} = A^{\frac{1}{2}} B A^{\frac{1}{2}};
\end{equation*}
that is, $AB$ is similar to a symmetric matrix and, thus,  it is diagonalizable. Accordingly, we have $AB = S^{-1} \Lambda S$, where $\Lambda = diag(\lambda_1,\lambda_2,\ldots,\lambda_d)$. Then,
\begin{equation*}
\tr(AB)^2 = ( \sum_i \lambda_i)^2 = \sum_{i,j} \lambda_i \lambda_j \leq \sum_{i,j} \frac{\lambda_i^2}{2} + \frac{\lambda_j^2}{2} = d \sum_i \lambda_i^2 = \tr\Big( (AB)^2 \Big).
\end{equation*}
\end{proof}

\bibliographystyle{plain}

\def\cprime{$'$}

\end{document}